\documentclass[a4paper,12pt]{amsart}
\usepackage{amsfonts,amscd,amsmath,amssymb}

\newcommand{\NN}{{\mathbb{N}}}

\newcommand{\bG}{{\mathbf{G}}}

\newcommand{\cB}{{\mathcal B}}
\newcommand{\cE}{{\mathcal E}}
\newcommand{\cH}{{\mathcal H}}

\newcommand{\fS}{{\mathfrak{S}}}
\newcommand{\fA}{{\mathfrak{A}}}

\newcommand{\Irr}{\operatorname{Irr}}
\newcommand{\GL}{{\operatorname{GL}}}
\newcommand{\SL}{{\operatorname{SL}}}
\newcommand{\PGL}{{\operatorname{PGL}}}
\newcommand{\SU}{{\operatorname{SU}}}
\newcommand{\SC}{{\operatorname{sc}}}
\newcommand\oh[1]{{{\bf O}_{#1}}}
\let\la=\lambda

\newcommand{\tw}[1]{{}^{#1}\!}

\newtheorem{thm}{Theorem}[section]
\newtheorem{lem}[thm]{Lemma}
\newtheorem{cor}[thm]{Corollary}
\newtheorem{prop}[thm]{Proposition}

\newtheorem*{teoA}{Theorem A}
\newtheorem*{teoB}{Theorem B}
\newtheorem*{teoC}{Theorem C}

\theoremstyle{definition}
\newtheorem{nota}[thm]{Notation}

\theoremstyle{remark}

\begin{document}

\title{Even degree characters in principal blocks}

\author{Eugenio Giannelli}
\address{Trinity Hall, University of Cambridge, Trinity Lane, CB21TJ, UK}
\email{eg513@cam.ac.uk}

\author{Gunter Malle}
\address{FB Mathematik, TU Kaiserslautern, Postfach 3049,
  67653 Kaiserslautern, Germany}
\email{malle@mathematik.uni-kl.de}

\author{Carolina Vallejo Rodr\'iguez}
\address{ICMAT, Campus Cantoblanco UAM, C/ Nicol\' as Cabrera, 13-15,
  28049 Madrid, Spain}
\email{carolina.vallejo@icmat.es}

\thanks{The first author's research is funded by Trinity Hall, University of
Cambridge. The second author gratefully acknowledges financial support by
SFB TRR 195. The third author is partially supported by the Spanish Ministerio
de Educaci\'on y Ciencia Proyectos MTM2016-76196-P and the ICMAT Severo Ochoa
project SEV-2011-0087.}



\begin{abstract}
We characterise finite groups such that for an odd prime $p$ all the
irreducible characters in its principal $p$-block have odd degree. We show
that this situation does not occur in non-abelian simple groups of order
divisible by~$p$ unless $p=7$ and the group is $M_{22}$. As a consequence we
deduce that if $p\neq 7$ or if $M_{22}$ is not a composition factor of a
group $G$, then the condition above is equivalent to $G/\oh{p'}(G)$ having
odd order.
\end{abstract}

\maketitle


\section{Introduction}

Let $G$ be a finite group, let $p$ be a prime dividing the order of $G$ and
let $B_0$ be the Brauer
principal ($p$-)block of $G$. Brauer's height zero conjecture asserts that $p$
does not divide the degrees of the irreducible ordinary characters belonging
to $B_0$ if, and only if, a Sylow $p$-subgroup $P$ of $G$ is abelian. Let $q$
be a prime different from $p$. It would be interesting to characterise when
all degrees of irreducible ordinary characters belonging to $B_0$ are coprime
with $q$. When $q=2$, G. Navarro predicted that all irreducible ordinary
characters in $B_0$ have odd degree if, and only if, $G/\oh{p'}(G)$ has odd
order. We confirm here that this claim holds whenever $p\neq 7$.
For $p=7$, the group $M_{22}$ is a counterexample, and the only counterexample
among finite simple groups.

\begin{teoA}
 Let $p$ be an odd prime, and let $B_0$ be the principal $p$-block of a group
 $G$ of order divisible by $p$. If $p\neq 7$ or $M_{22}$ is not a composition
 factor of $G$, then every irreducible character in $B_0$ has odd degree if,
 and only if, $G/\oh{p'}(G)$ is a group of odd order.
\end{teoA}

It is worth mentioning that G.~Navarro, G.~Robinson and P.H.~Tiep \cite{NRT17}
have recently characterised the groups for which $G/\oh {p'}(G)$ has odd order
in terms of the number of real-valued characters in the principal $p$-block
of $G$. Namely they have shown that the principal $p$-block $B_0$ contains only
one real-valued irreducible character if, and only if, $G/\oh{p'}(G)$ has odd
order.

\smallskip

The proof of Theorem A ultimately relies on the fact that for every finite
non-abelian simple group $G$ and every odd prime divisor $p$ of its order,
the principal $p$-block of $G$ contains an even degree character, except for
$p=7$ and $G=M_{22}$. 

\begin{teoB}
 Let $G$ be a non-abelian simple group of order divisible by an odd prime $p$.
 Let $B_0$ be the principal $p$-block of $G$. Then $B_0$ contains an
 irreducible character $\chi$ of even degree unless $G=M_{22}$ and $p=7$.
\end{teoB}

In order to show that Theorem B holds for alternating groups we prove a more
general result that we believe to be of independent interest, concerning the
principal block of symmetric (and alternating) groups.

\begin{teoC}
 Let $n\geq 5$, and let $p$ and $q$ be distinct primes with $q<p\leq n$.
 Let $B_0$ be the principal $p$-block of either $\fS_n$ or $\fA_n$. Then
 $B_0$ contains an irreducible character $\chi$ of degree divisible by $q$.
\end{teoC}

This article is structured as follows. In Section 2 we prove that Theorem~A
holds assuming Theorem~B. The rest of the article is devoted to the proof of
Theorem~B on finite non-abelian simple groups. In Section~3 we prove Theorem~C,
which includes the alternating group case of Theorem~B. Finally, in Section~4
we show that Theorem~B holds for sporadic groups and for groups of Lie type,
and conclude using the Classification of Finite Simple Groups.

\smallskip

We follow the notation of \cite{Is76} for ordinary characters and the
notation of \cite{Nav98} for blocks. By a block, we shall mean a $p$-block.
Also, if $B$ is a block of $G$, we will denote by $\Irr(B)$ the set of
irreducible complex characters lying in the block $B$. In general, we will
denote by $B_0(H)$ the principal block of a group $H$.
Finally if $n$ is a natural number, then we denote by $\nu_p(n)$ the maximal
integer $k$ such that $p^k$ divides $n$.
\medskip

\noindent
{\bf Acknowledgement.}~~Part of this work was done while the third author was
visiting the Department of Mathematics at the TU Kaiserslautern. She would
like to thank everyone at the department for the warm hospitality. The authors
would also like to thank Gabriel Navarro for his insight. 

\medskip

\section{A reduction to finite simple groups}\label{sec 2}
In this section we assume that Theorem B holds and we prove Theorem A. In this
sense, we reduce Theorem A to a problem on finite simple groups.

\smallskip

We assume that the reader is familiar with the theory of blocks and normal
subgroups (see for example Chapter 9 of \cite{Nav98}).
Also, we recall that if $G_1$ and $G_2$ are finite groups, then
$\Irr(B_0(G_1\times G_2))=\Irr(B_0(G_1))\times \Irr(B_0(G_2))$ (this follows
directly from the definition of the principal block, see
\cite[Def.~3.1]{Nav98}).

\begin{proof}[Proof of Theorem A]
We remark that $\Irr(B_0)\subseteq \Irr(G/\oh{p'}(G))$, since $B_0$ covers the
principal block of $\oh{p'}(G)$ which only consists of the trivial character.
Also, whenever $N\unlhd G$, we have that $\Irr(B_0(G/N))\subseteq \Irr(B_0)$.

The ``if part" follows from the first remark. To prove the ``only if part" we
first show that the hypotheses imply the $p$-solvability of $G$. Otherwise,
since the hypotheses are inherited by quotients of $G$, we may assume that
there is a minimal normal subgroup $N$ of $G$ which is neither a $p$-group nor
a $p'$-group. Hence $N=\prod_{i=1}^t S^{x_i}$, where $S\unlhd N$ is a simple
non-abelian group of order divisible by $p$ and $x_i \in G$. By hypothesis,
either $p\neq 7$ or $S\neq M_{22}$. By Theorem B, let $\theta \in \Irr(B_0(S))$
have even degree. Then $\phi=\theta^{x_1}\times \cdots \times \theta^{x_t}$
has even degree and belongs to
$\Irr(B_0(N))=\prod_{i=1}^t \Irr(B_0(S))^{x_i}$. By \cite[Thm.~9.4]{Nav98},
some $\chi \in \Irr(B_0)$ lies over $\phi$. Since $\chi(1)$ is odd by
hypothesis, we get a contradiction.

As $G$ is $p$-solvable, we have that $\Irr(B_0)=\Irr(G/\oh{p'}(G))$, by
\cite[Thm.~10.20]{Nav98}. Hence, the hypothesis that all irreducible characters
in $B_0$ have odd degree implies that the group $G/\oh{p'}(G)$ has a normal
Sylow $2$-subgroup by the Ito--Michler theorem (see \cite{It51} and
\cite{Mi86}). This forces $G/\oh{p'}(G)$ to be a group of odd order, as
desired.
\end{proof}

Recall that our motivation is to characterise when a prime $q$ does not divide
the degrees of the irreducible characters in the principal block. The
Ito--Michler theorem characterises when a prime $q$ does not divide the
degrees of the irreducible characters of a group. A natural version of the
Ito--Michler theorem for principal blocks would be: If all the irreducible
characters of $B_0(G)$ have degree coprime to $q$, for some prime $q\neq p$,
then some Sylow $q$-subgroup $Q$ of $G$ is normalised by a Sylow $p$-subgroup
$P$ of $G$. In \cite{NW01} the authors prove this result under the assumption
that $G$ is a $\{p, q\}$-separable group. However, such a version does not
hold outside $\{p,q\}$-separable groups, as the authors also point out that
the separability condition of $G$ is necessary (as shown by $G=J_1$, $p=2$,
$q=5$). For $q=2$ and $p\neq 7$ we have the characterisation given by
Theorem~A.


\section{Alternating groups}
The aim of this section is to prove Theorem C, which in particular implies
Theorem B for alternating groups.

\smallskip

We start by recalling some facts in the representation theory of symmetric
groups. We refer the reader to \cite{James}, \cite{JK} or \cite{OlssonBook}
for a more detailed account.
A partition $\la=(\la_1,\la_2,\dots,\la_\ell)$ is a finite
non-increasing sequence of positive integers. We say that $\la$ is a
partition of $|\la|=\sum\la_i$, written $\la\vdash|\la|$.
The Young diagram of $\la$ is the set $[\la]=\{(i,j)\in\NN\times\NN\mid
1\leq i\leq\ell,1\leq j\leq\la_i\},$ where we orient $\NN\times\NN$ with
the $x$-axis pointing right and the $y$-axis pointing down. We denote by
$\la'$ the conjugate partition of $\la$, whose Young diagram is
obtained from the Young diagram of $\la$ by a reflection over the main
diagonal.

\smallskip

Given $(r,c)\in[\la]$, the corresponding hook $H_{(r,c)}(\la)$ is the
set defined by
$$H_{(r,c)}(\la)
  =\{(r,y)\in [\la]\mid y\geq c\}\cup\{(x,c)\in [\la]\mid x\geq r\}.$$
We set $h_{r,c}(\la)=|H_{(r,c)}(\la)|=1+(\la_r-c)+(\la'_c-r)$,
and we denote by $\cH(\la)$ the multiset of hook-lengths in $[\la]$.
Similarly we let $\cH_r(\la)$ be the multiset of hook-lengths in the
$r$th row of $[\la]$. For $e\in\NN$ we let
$\cH^e(\la)=\bigcup_{r=1}^{\ell}\cH_r^e(\la)$, where
$\cH_r^e(\la)=\{(r,c)\in [\la]\mid e \text{ divides }h_{r,c}(\la)\}$.
If $(r,c)\in\cH^e(\la)$, then we say that $H_{(r,c)}(\la)$ is an $e$-hook
of $\la$, so that $|\cH ^e (\la)|$ is the number of $e$-hooks of $\la$.
The $e$-core $C_e(\la)$ of $\la$ is the partition obtained from $\la$ by
successively removing all $e$-hooks.

\smallskip

Let $q$ be a prime number.
An important combinatorial object for our analysis is the $q$-core tower
$T^q(\la)$ of a partition $\la$ of $n$ (we refer the reader to
\cite[Chap.~II]{OlssonBook} for a comprehensive description of this object).
Every partition of a given natural number is uniquely determined by its
$q$-core tower. We write $T(\la)$ instead of $T^q(\la)$ when $q$ is clear
from the context. For $j\in\NN_0$ we denote by $T_j(\la)$ the $j$th layer of
$T(\la)$. As explained in full details in \cite{OlssonBook},
$T_j(\la)=(\mu_1,\ldots, \mu_{q^j})$ is a sequence of $q$-core partitions such
that
$$\sum_{j\geq 0}|T_j(\la)|q^j=n,\ \ \  \text{where}\ \ \ 
   |T_j(\la)|=|\mu_1|+\cdots +|\mu_{q^j}|.$$
Note that $|T_0(\la)|=|C_q(\la)|$. Also note that if
$n=\sum_{j=0}^k\alpha_jq^j$ is the $q$-adic expansion of $n$, then
$|T_s(\la)|=0$ whenever $s>k$.
It is useful to remark at this point that if $m$ is the maximal integer such
that $|T_m(\la)|\neq 0$, then we have that $|\cH^{q^m}(\la)|=|T_m(\la)|$.

\smallskip

Partitions of $n$ correspond canonically to the irreducible characters of
$\fS_n$. We denote by $\chi^\la$ the irreducible character naturally labelled
by $\la\vdash n$. We recall that $(\chi^\la)_{\fA_n}$ is irreducible if,
and only if, $\la\neq \la'$. Otherwise $(\chi^\la)_{\fA_n}=\phi+\phi^{g}$ for
some $\phi\in\Irr(\fA_n)$ and $g\in\fS_n\smallsetminus\fA_n$.
(See \cite[Thm.~2.5.7]{JK}.)

\smallskip

The following result was first proved by MacDonald \cite{Mac} and it is
crucial for our purposes.

\begin{thm}[MacDonald]   \label{theo: mac}
 Let $q$ be a prime and let $n$ be a natural number with $q$-adic expansion
 $n=\sum_{j=0}^k\alpha_jq^j$. Let $\la$ be a partition of $n$. Then
 $$\nu_q(\chi^\la(1))
   =\big(\sum_{j\geq 0}|T_j(\la)|-\sum_{j=0}^k\alpha_j\big)/(q-1).$$
\end{thm}

Useful consequences of this result are recorded in the following two lemmas.

\begin{lem}   \label{lem: last layer}
 Let $q, n$ and $\la$ be as in Theorem~\ref{theo: mac}. If $q$ does not divide
 $\chi^\la(1)$ then $|\cH^{q^k}(\la)|=\alpha_k$.
\end{lem}

\begin{lem}   \label{lem: k+1-alpha}
 Let $k\geq 2$ and let $n=2^{k+1}-2^\ell$, for some $\ell\leq k-2$. Let
 $\la\vdash n$ such that $|\cH^{2^k}(\la)|=0$ and $|\cH^{2^{k-1}}(\la)|\leq 2$.
 Then $\nu_2(\chi^\la(1))\geq 2$.
\end{lem}

\begin{proof}
Since $|\cH^{2^k}(\la)|=0$ we deduce that $|T_k(\la)|=0$. Hence we have that
$|T_{k-1}(\la)|=|\cH^{2^{k-1}}(\la)|\leq 2$.
We know that $n=\sum_{j=0}^{k}|T_j(\la)|2^j$, hence we obtain that
$\sum_{j=0}^{k}|T_j(\la)|\geq k+3-\ell$. Since $n$ has $k+1-\ell$ binary
digits, we deduce that $\nu_2(\chi^\la(1))\geq 2$ from Theorem~\ref{theo: mac}.
\end{proof}

Finally, we state the following consequence of Nakayama's Conjecture (proved
independently by R.~Brauer and G.~de B.~Robinson, see e.g. \cite{James} or
\cite{OlssonBook}).
The second statement follows from \cite[Thm.~9.2]{Nav98}.

\begin{prop}   \label{prop: Nak}
 Let $p$ be a prime and let $n=a+pw$ for some $a\in\{0,1,\ldots, p-1\}$ and
 $w\in\NN$. Let $\la\vdash n$. Then $\chi^\la$ lies in the principal $p$-block
 of $\fS_n$ if, and only if, $C_p(\la)=(a)$. Moreover every irreducible
 constituent of $(\chi^\la)_{\fA_n}$ lies in the principal $p$-block of $\fA_n$.
\end{prop}

In order to prove Theorem C, we devote most of our efforts to show that the
following slightly stronger statement holds.

\begin{thm}   \label{teo: Sym q}
 Let $q$ and $p$ be primes and let $n\geq 5$ be a natural number such that
 $q<p\leq n$. Then there exists an irreducible character $\chi$ in the
 principal $p$-block of $\fS_n$ such that $q$ divides $\chi(1)$. Moreover,
 if $\phi$ is an irreducible constituent of the restriction $\chi_{\fA_n}$,
 then $q$ divides $\phi(1)$.
\end{thm}

We split the proof of Theorem~\ref{teo: Sym q} into a series of lemmas (dealing
with symmetric groups) and corresponding corollaries (concerning alternating
groups). We fix below the notation that will be kept for the rest of this
section.

\begin{nota}
Let $2\leq q<p$ be prime numbers and let $n\in\NN$, which we write as $n=a+pw$
for uniquely determined $a\in\{0,1,\ldots, p-1\}$ and $w\in\NN$. Moreover let
$$n=\alpha_kq^k+\sum_{j=0}^{k-1}\alpha_jq^j,$$
be the $q$-adic expansion of $n$, where $\alpha_k\neq 0$.
We need also to fix the notation for the $q$-adic expansions of $pw$ and $a$.
We let
$$a=\beta_kq^k+\sum_{j=0}^{k-1}\beta_jq^j\ \ \ \text{and} \ \ \ 
  pw=\zeta_kq^k+\sum_{j=0}^{k-1}\zeta_jq^j,$$
where $\beta_j, \zeta_j\in\{0,1,\ldots, q-1\}$ for all $j\in\{0,\ldots, k-1\}$
and $0\leq\beta_k\leq\zeta_k\leq \alpha_k\leq q-1$. Note that
$pw\neq \zeta_k q^k$, as both $\zeta_k$ and  $q$ are smaller than $p$.

\smallskip

Finally we denote by $\cB_n(q,p)$ the set consisting of all partitions $\la$
of $n$ such that $\chi^\la$ lies in the principal $p$-block of $\fS_n$ and
such that $q$ divides $\chi^\la(1)$.
\end{nota}

\begin{lem}\label{lem: a=0}
 Suppose that $a=0$. Then $\la=(\alpha_kq^k, 1^{n-\alpha_kq^k})\in\cB_n(q,p)$.
\end{lem}

\begin{proof}
It is easy to see that $C_p(\la)$ is the empty partition. Therefore $\chi^\la$
lies in the principal block of $\fS_n$, by Proposition~\ref{prop: Nak}.
The hook-length $h_{1,1}(\la)=wp$ is not divisible by $q^k$. Moreover, for
all $j\in\{2,3,\ldots, \alpha_kq^k\}$ we have
$h_{1,j}(\la)+(j-2)=h_{1,2}(\la)=\alpha_kq^k-1$. Hence we deduce that there
are exactly $\alpha_k-1$ boxes $(1, d)\in [\la]$ such that $q^k$ divides
$h_{1,d}(\la)$. It follows that $|\cH_1^{q^k}(\la)|=\alpha_k-1$. Since
$h_{2,1}(\la)=n-\alpha_kq^k<q^k$ we deduce that $|\cH_j^{q^k}(\la)|=0$ for
all $j\geq 2$. We conclude that $|\cH^{q^k}(\la)|=\alpha_k-1$ and therefore
that $q$ divides $\chi^\la(1)$, by Lemma~\ref{lem: last layer}.
\end{proof}

We have preferred to show in our proof of Lemma \ref{lem: a=0} the strategy
that will be used to prove Lemmas~\ref{lem: pw>} and~\ref{lem: pw<} below,
instead of using the well-known fact that whenever $\mu=(n-d, 1^d)$ is a hook
partition the degree $\chi^\mu(1)$ is equal to the binomial coefficient
$\binom{n-1}{d}$. In any case, the necessary analysis of the divisibility
of the binomial coefficient would not have sensibly shortened our argument.

\begin{cor}   \label{cor: a1}
 Let $a$ and $\la$ be as in Lemma \ref{lem: a=0} and let $\phi$ be an
 irreducible constituent of $(\chi^\la)_{\fA_n}$. Then $\phi$ lies in the
 principal $p$-block of $\fA_n$ and $q$ divides $\phi(1)$.
\end{cor}

\begin{proof}
Since $\phi(1)\in\{\chi^\la(1),\chi^\la(1)/2\}$ the statement follows trivially
from Lemma \ref{lem: a=0}, when $q\neq 2$.
Similarly if $q=2$ and $\la\neq\la'$ then, $\phi=(\chi^\la)_{\fA_n}$; and
hence the statement follows, again by Lemma~\ref{lem: a=0}.

Suppose that $\la=(2^k, 1^{n-2^k})=\la'$. Then $n=2^{k+1}-1$
and we observe that $|\cH^{2^k}(\la)|=0$ and that $|\cH^{2^{k-1}}(\la)|=2$.
Using Lemma \ref{lem: k+1-alpha} we deduce that $\nu_2(\chi^\la(1))\geq 2$ and
therefore that $\phi(1)=\chi^\la(1)/2$ is even.
\end{proof}

From now on, we assume that $1\leq a\leq p-1$.

\begin{lem}   \label{lem: pw>}
 Suppose that $wp>\alpha_kq^k$ and let
 $\la=(\alpha_kq^k-1,a+1,1^{n-(\alpha_kq^k+a)})$. Then $\la\in\cB_n(q,p)$ for
 all $n\neq (\alpha_k+1)q^k-1$. For $n= (\alpha_k+1)q^k-1$, we construct
 $\mu\in\cB_n(q,p)$ as follows:
 $$\mu= \begin{cases}
   (a,2,1^{p-2}) & \mathrm{if}\ q=2, w=1\ \ \text{and}\ \ a=p-3,\\
   (\alpha_kq^k-2, a+1, 1^{n-(\alpha_kq^k+a-1)}) & \mathrm{otherwise}.
 \end{cases}$$
\end{lem}

\begin{proof}
Let $n\neq (\alpha_k+1)q^k-1$. Note that $h_{1,1}(\la)=n-a=wp$. This implies
that $C_p(\la)=(a)$. Moreover, $h_{1,1}(\la)$ is not divisible by $q^k$.
Analysing other hook-lengths in $[\la]$, we see that for all
$j\in\{3,\ldots, \la_1\}$ we have
$1\leq h_{1,j}(\la)<h_{1,2}(\la)=\alpha_kq^k-1$. It follows that
$|\cH_1^{q^k}(\la)|\leq \alpha_k-1$.
On the other hand, $h_{2,1}(\la)=n-(\alpha_kq^k-1)<q^k$ and therefore
$|\cH_j^{q^k}(\la)|=0$ for all $j\geq 2$. We deduce that
$|\cH^{q^k}(\la)|\leq \alpha_k-1$ and hence that $q$ divides $\chi^{\la}(1)$,
by Lemma~\ref{lem: last layer}. Notice that for $n=(\alpha_k+1)q^k-1$ we have
that $h_{2,1}(\la)=q^k$. For this reason we need to construct a different
partition in this case.

Let $\mu=(\alpha_kq^k-2, a+1, 1^{n-(\alpha_kq^k+a-1)})$. Then the composition
$\mu$ defines a partition of $n$ unless $\mu_2>\mu_1$. This happens if, and
only if, $a\geq \alpha_kq^k-2$ and in turn this is equivalent to the following
chain of inequalities:
$$n=a+pw\geq (\alpha_kq^k-2)+(\alpha_kq^k+1)=n+(\alpha_k-1)q^k\geq n.$$
This is consistent if, and only if, we have equalities everywhere. Equivalently
we must have $\alpha_k=1$, $pw=q^k+1$ and $a=q^{k}-2=pw-3$.
Moreover, it is now easy to see that $pw-a=3$ forces $w=1$. Hence we have that
$p=q^k+1$, which implies that $q=2$.
The above discussion shows that $\mu_2>\mu_1$ if, and only if, $q=2, w=1$ and
$a=p-3$. In this case we set $\mu=(a,2,1^{p-2})$.
Observing that $h_{2,1}(\mu)=q^k+1$ and
arguing exactly as before we verify that $\mu \in\cB_n(q,p)$ (in both cases).
\end{proof}

\begin{cor}   \label{cor: a2}
 Let $a, wp, \la$ and $\mu$ be as in Lemma~\ref{lem: pw>}. Let
 $\delta\in\{\la, \mu\}$ and let $\phi$ be an irreducible constituent of
 $(\chi^\delta)_{\fA_n}$. Then $\phi$ lies in the principal $p$-block of
 $\fA_n$ and $q$ divides $\phi(1)$.
\end{cor}

\begin{proof}
Exactly as in the proof of Corollary \ref{cor: a1}, we can focus on the case
where $q=2$. If $n=2^{k+1}-1$ then $\delta=\mu\neq \mu'$ and therefore
$\chi^\mu(1)=\phi(1)$ is even. If $n\neq 2^{k+1}-1$ then $\la=\la'$
if and only if $a=1$ and $n=2^{k+1}-2$. Analysing the hook-lengths in
$[\la]$ we observe that $|\cH^{2^k}(\la)|=0$ and that
$|\cH^{2^{k-1}}(\la)|\leq 2$.
Using Lemma~\ref{lem: k+1-alpha} we deduce that $\nu_2(\chi^\la(1))\geq 2$
and therefore that $\phi(1)=\chi^\la(1)/2$ is even.
\end{proof}

\begin{lem}   \label{lem: pw<}
 Suppose that $wp<\alpha_kq^k$. If $\beta_k>0$, then define $\la\vdash n$
 as follows:
 $$\la=(a,a-\beta_kq^k+1, 1^{n-2a+\beta_kq^k-1}).$$
 Otherwise let $\la\vdash n$ be defined by:
 $$\la= \begin{cases}
   (n-2,2) & \mathrm{if}\ a=1,\\
   (a,1^{wp}) & \mathrm{if}\ a>1\ \ \text{and}\ \ n\neq \alpha_kq^k,\\
   (a,2,1^{wp-2}) & \mathrm{if}\ a>1\ \ \text{and}\ \ n= \alpha_kq^k.
 \end{cases}$$
 In all cases we have that $\la\in\cB_n(q,p)$.
\end{lem}

\begin{proof}
Let us first analyse the case where $\beta_k>0$. It is clear that $\la$ is
well defined and that we have $C_p(\la)=(a)$ (since $h_{2,1}(\la)=wp$).
Moreover we observe that $|\cH_1^{q^k}(\la)|=\beta_k-1$, $|\cH_2^{q^k}(\la)|=0$
and that $\sum_{j\geq 3}|\cH_j^{q^k}(\la)|\leq\zeta_k$. It follows that
$|\cH^{q^k}(\la)|\leq\alpha_k-1$ and hence that $q$ divides
$\chi^{\la}(1)$, by Lemma~\ref{lem: last layer}.

We assume now that $\beta_k=0$ and let $\la=(a,1^{wp})$. Clearly
$C_p(\la)=(a)$. Moreover, analysing again the hook-lengths in $[\la]$ and
using Lemma \ref{lem: last layer} we deduce that $q$ divides $\chi^\la(1)$
unless $a=1$ or $n=\alpha_kq^k$. In these two cases we choose $\la$ as
described in the statement and we verify that $C_p(\la)=(a)$ and that $q$
divides $\chi^\la(1)$, again using Lemma~\ref{lem: last layer}.
\end{proof}

\begin{cor}   \label{cor: a3}
 Let $a, pw$ and $\la$ be as in Lemma \ref{lem: pw<}. Let $\phi$ be an
 irreducible constituent of $(\chi^\la)_{\fA_n}$. Then $\phi\in B_0(\fA_n)$
 has degree divisible by $q$.
\end{cor}

\begin{proof}
As usual, if $q\neq 2$ then the statement follows easily from
Lemma~\ref{lem: pw<}. When $q=2$ the parameter $\beta_k$ is the coefficient
of $2^k$ in the binary expansion of $a$. Since $a<p\leq wp<n$ we deduce that
$\beta_k<\alpha_k\leq 1$, hence $\beta_k=0$. Moreover, as usual we only need
to analyse the situations where $\la=\la'$. It is easy to check that this
never occurs. Hence $\chi^\la(1)=\phi(1)$ is even.
\end{proof}

\begin{proof}[Proof of Theorem C]
The statement of Theorem~\ref{teo: Sym q} concerning $\fS_n$ follows from
Lemmas~\ref{lem: a=0}, \ref{lem: pw>} and~\ref{lem: pw<}. On the other hand,
the part of Theorem~\ref{teo: Sym q} concerning $\fA_n$ follows from
Corollaries~\ref{cor: a1}, \ref{cor: a2} and~\ref{cor: a3}.
As already remarked at the beginning of the section, Theorem~\ref{teo: Sym q}
implies Theorem C.
\end{proof}

In order to prove Theorem C without the assumption $q<p$, we would need a
radically different combinatorial approach. This is not directly relevant for
the purpose of the present article. Nevertheless, it could be material for
further investigation.

\section{Simple groups of Lie type and sporadic groups}

The aim of this section is the proof of Theorem~B on the principal blocks for
the remaining non-abelian simple groups. The case of sporadic groups is
immediately checked from the known character tables, leading to the unique
counter-example $M_{22}$ at the prime~7, so by the classification of the
finite simple groups we are left to deal with the finite simple group of Lie
type.

\begin{thm}   \label{thm:lietype}
 Let $S$ be a finite simple group of Lie type and $\ell$ an odd prime dividing
 $|S|$. Then the principal $\ell$-block of $S$ contains an irreducible
 character of even degree.
\end{thm}

\begin{proof}
We discuss the various cases. First assume that the defining characteristic $p$
of $S$ coincides with the (odd) prime $\ell$. By a well-known result then
the only irreducible character of $S$ not contained in the principal $p$-block
is the Steinberg character, whose degree is a power of $p$, hence in particular
odd. The claim in this case now follows from the fact that a non-abelian
simple group has order divisible by at least three distinct primes together
with the remaining part of the proof. (Alternatively, a non-abelian simple
group possesses an irreducible character of even degree by a result of Willems.
By what we said before that character will lie in the principal $p$-block.)
\par
We may hence assume that $\ell$ is not the defining characteristic of $S$.
So $S$ is not a Suzuki group nor one of the big Ree groups. We now set up the
following notation. Let $\bG$ be a simple algebraic group of simply connected
type over an algebraic closure of the finite field with $p$ elements, and
$F:\bG\rightarrow\bG$ a Steinberg map, such that $S=G/Z(G)$, where $G:=\bG^F$.
\par
By \cite[Thm.]{CE94} the principal $\ell$-block of $G$ lies in the union of
Lusztig series
$$\cE_\ell(G,1)=\coprod_t \cE(G,t)$$
with $t$ running over the (semisimple) $\ell$-elements of the dual group
$G^*$ modulo conjugation, and in particular, it contains the semisimple
character(s) $\chi_t$ from each series $\cE(G,t)$. Now the degree of $\chi_t$
is given by
$$\chi_t(1)=|G^*:C_{G^*}(t)|_{p'}.$$
In particular this provides a character of even degree in the principal
$\ell$-block of $G$ if $t$ does not centralise a Sylow 2-subgroup of $G^*$.
This is an even degree character in the principal $\ell$-block of the simple
group $S=G/Z(G)$ if it has $Z(G)$ in its kernel, which happens if and only if
$t\in[G^*,G^*]$. Thus, to conclude we need to construct an $\ell$-element
$t\in [G^*,G^*]\cong S$ not centralising a Sylow 2-subgroup of $G^*$.
\par
First assume that $\ell$ does not divide the order of the centraliser of a
Sylow $2$-subgroup of $S$. Then clearly any non-trivial $\ell$-element of $S$
is as required. Now by \cite[Thm.~7]{KM03} a Sylow 2-subgroup $P$ of $S$ is
self-centralising, unless $G$ is one of $\SL_n(q)$, $\SU_n(q)$, $E_6(q)_\SC$
or $\tw2E_6(q)_\SC$. So it remains to consider these four families of groups.
First assume that $G=\SL_n(q)$,
and let $2^{t_1}+\dots+2^{t_r}$, with $t_1<\ldots<t_r$ be the 2-adic
expansion of $n$. Then $PC_S(P)/P$ is a direct product of at most $r-1$ cyclic
groups of order dividing $q-1$, so we may assume that $r\ge2$ and $\ell|(q-1)$.
Now let $t\in\GL_n(q)$ be an $\ell$-element with two distinct eigenvalues, one
with multiplicity $2^{t_r}+1$, the other with $n-2^{t_r}+1$ (respectively one
with multiplicity $2^{t_r}-1$, the other with multiplicity~2 if $r=2$ and
$t_1=0$; respectively with three distinct eigenvalues when $n=3$). Then the
image of $t$ in $G^*=\PGL_n(q)$ does not centralise a Sylow 2-subgroup of $G^*$.
If $G=\SU_n(q)$, then again $PC_S(P)/P$ is a direct product of at most $r-1$
cyclic groups, of order dividing $q+1$ this time, with $r$ as before. Hence
in this case $\ell$ divides $q+1$, and we can construct an element $t$ as
before.
\par
For $S=E_6(q)$ the centraliser of a Sylow 2-subgroup contains a cyclic subgroup
of order dividing $q-1$. Hence again $\ell$ divides $q-1$. Now take $t$ an
$\ell$-element in the centre of a Levi subgroup $L$ of $S$ of type $A_5$. Then
$L$ is the full centraliser of $t$ (as $L$ is a maximal reductive subgroup
of $S$), but its index in $S$ is even, so $t$ is as required. Finally, for
$S=\tw2E_6(q)$, we need to consider divisors $\ell$ of $q+1$, and here any
$\ell$-element in the centre of a Levi subgroup of twisted type $\tw2A_5$
will do.
The proof is complete.
\end{proof}


\end{document}